\documentclass[11pt]{amsart}
\usepackage{amssymb,amsmath,amsthm,newlfont,enumerate}

\usepackage{graphicx}
\usepackage{amscd}
\usepackage{amsfonts}
\usepackage{mathrsfs}
\usepackage{bm}
\usepackage{enumerate}
\usepackage{amsrefs}
\usepackage{xcolor}
\usepackage[colorlinks, citecolor=blue, linkcolor=red, pdfstartview=FitB]{hyperref}
\usepackage[latin1]{inputenc}
\theoremstyle{plain}
\newtheorem{theorem}{Theorem}[section]
\newtheorem{proposition}[theorem]{Proposition}
\newtheorem{lemma}[theorem]{Lemma}

\newtheorem{conjecture}[theorem]{Conjecture}

\theoremstyle{definition}

\theoremstyle{remark}
\newtheorem*{remark}{Remark}
\newtheorem*{remarks}{Remarks}

\newcommand{\ol}{\overline}

\newcommand{\CC}{{\mathbb C}}
\newcommand{\DD}{{\mathbb D}}

\newcommand{\TT}{{\mathbb T}}

\newcommand{\cC}{{\mathcal C}}

\newcommand{\cH}{{\mathcal H}}
\newcommand{\cK}{{\mathcal K}}

\renewcommand{\Re}{\operatorname{Re}}

\DeclareMathOperator{\conv}{conv}

\let\Re\undefined
\DeclareMathOperator{\Re}{Re}

\renewcommand{\tilde}{\widetilde}

\begin{document}

\date{\today}

\title{An abstract approach to the Crouzeix conjecture}

\author{Rapha\"el Clou\^atre}
\address{Department of Mathematics, University of Manitoba,
186 Dysart Road, Winnipeg (Manitoba), Canada R3T 2N2.}
\email{raphael.clouatre@umanitoba.ca}
\thanks{Clou\^atre partially supported by an NSERC Discovery Grant.}

\author{Ma\"eva Ostermann}
\address{D\'epartement de math\'ematiques et de statistique, Universit\'e Laval,
Qu\'ebec City (Qu\'ebec),  Canada G1V 0A6.}
\email{maeva.ostermann.1@ulaval.ca}
\thanks{Ostermann supported by a FRQNT doctoral scholarship}

\author{Thomas Ransford}
\address{D\'epartement de math\'ematiques et de statistique, Universit\'e Laval,
Qu\'ebec City (Qu\'ebec),  Canada G1V 0A6.}
\email{thomas.ransford@mat.ulaval.ca}
\thanks{Ransford supported by grants from NSERC and the Canada Research Chairs program.}

\begin{abstract}
Let $A$ be a uniform algebra, $\theta:A\to M_n(\CC)$ be a continuous homomorphism
and $\alpha:A\to A$ be an antilinear contraction such that
\[
\|\theta(f)+\theta(\alpha(f))^*\|\le 2\|f\| \quad(f\in A).
\]
We show that $\|\theta\|\le 1+\sqrt{2}$, and that $1+\sqrt2$ is sharp. 
We  conjecture that, if further  $\alpha(1)=1$, then
we may  conclude that $\|\theta\|\le2$. This would 
yield a positive solution to the Crouzeix conjecture on numerical ranges. 
In support of our conjecture,
we prove that it is true in two special cases.
We also discuss a completely bounded version of our conjecture that brings into play
ideas from dilation theory.
\end{abstract}



\maketitle


\section{Introduction}\label{S:intro}

In this article, we propose an approach to a conjecture about numerical ranges,
the so-called Crouzeix conjecture, through the study of certain homomorphisms on uniform algebras. Recall that given a compact Hausdorff space $X$, by a uniform algebra we mean a unital norm-closed subalgebra $A\subset C(X)$.  Throughout the paper, homomorphisms of are always assumed to be unital. Further, we denote by $M_n(\CC)$ the
algebra of complex $n\times n$ matrices.

Our starting point is the following  theorem.

\begin{theorem}\label{T:basic}
Let $A$ be a uniform algebra,
let $\theta:A\to M_n(\CC)$ be a continuous homomorphism,
and let $\alpha:A\to A$ be an antilinear contraction.
Define a linear map $\theta_\alpha:A\to M_n(\CC)$ by
\begin{equation}\label{E:thetaalpha}
\theta_\alpha(f):=\frac{1}{2}\Bigl(\theta(f)+\theta(\alpha(f))^*\Bigr) \quad(f\in A).
\end{equation}
If $\|\theta_\alpha\|\le 1$, then
\begin{equation}\label{E:1+sqrt2}
\|\theta\|\le 1+\sqrt{2}.
\end{equation}
Moreover, the constant $(1+\sqrt{2})$ is sharp.
\end{theorem}

We shall establish this result in \S\ref{S:basic}.
The example demonstrating the sharpness of the constant $(1+\sqrt{2})$
has the feature that $\alpha(1)=-1$.
We believe that if $\alpha$ is at the other extreme, namely if it satisfies $\alpha(1)=1$,
then Theorem~\ref{T:basic} may be improved as follows.

\begin{conjecture}\label{T:conj}
Let $A$ be a  uniform algebra,  $\theta:A\to M_n(\CC)$ be a continuous
homomorphism,
and let $\alpha:A\to A$
be a \textbf{unital} antilinear contraction.
If $\theta_\alpha$, defined by \eqref{E:thetaalpha}, satisfies $\|\theta_\alpha\|\le1$, then
\begin{equation}\label{E:2}
\|\theta\|\le 2.
\end{equation}
\end{conjecture}

Our interest in this conjecture is prompted  in part by the fact that
a positive solution, even in a very special case,
would lead to a proof of a celebrated conjecture of Crouzeix \cite{Cr04}. 
The latter conjecture states that,
if $T$ is an $n\times n$ matrix with numerical range $W(T)$, then,
for all polynomials $p$,
\begin{equation}\label{E:Crouzeix}
\|p(T)\|\le 2\max_{z\in W(T)}|p(z)|.
\end{equation}
In other words, $W(T)$ is a $2$-spectral set for $T$. 
For background on the Crouzeix conjecture and various partial results, 
we refer the reader to \cites{BCD06, BG20,  CGL18, Ch13, Cr04, Cr07, Cr16, CP17, GKL18, GO18}.

The link between the two conjectures is made precise by the following theorem.
We write $\DD$ for the open unit disk and $A(\DD)$ for the disk algebra,
namely the algebra of continuous functions on $\overline{\DD}$ that are holomorphic on~$\DD$.

\begin{theorem}\label{T:link}
If Conjecture~\ref{T:conj} holds when $A=A(\DD)$  and $n=n_0$,
then \eqref{E:Crouzeix} holds for all $T\in M_{n_0}(\CC)$ and all polynomials $p$.
\end{theorem}

This theorem will be proved in \S\ref{S:link}.
We note in passing that a similar argument, but using Theorem~\ref{T:basic}
instead of Conjecture~\ref{T:conj}, leads to the conclusion that \eqref{E:Crouzeix}
holds with $(1+\sqrt{2})$ in place of the constant~$2$.
This was already known: it is a result of Crouzeix and Palencia \cite{CP17}.
Indeed, the Crouzeix--Palencia theorem was the main inspiration for the present article.

We establish two partial results in support of Conjecture~\ref{T:conj}.
In the first of these, we add the assumption that $\alpha$ maps $A$ to 
multiples of the identity.
This  is an abstraction of an idea of Caldwell, Greenbaum and Li \cite[\S6]{CGL18},
and, as they showed, it implies the Crouzeix conjecture for the disk, 
a result  originally due to Okubo and Ando \cite{OA75}.

\begin{theorem}\label{T:CGL}
Let $A$ be a uniform algebra, let $\theta:A\to M_n(\CC)$ be a continuous homomorphism,
and let $\alpha:A\to A$ be a  unital antilinear  contraction 
such that $\alpha(A)\subset\CC1$.  
If $\theta_\alpha$ defined by \eqref{E:thetaalpha} satisfies $\|\theta_\alpha\|\le1$, then
$\|\theta\|\le2$.
\end{theorem}

In the second result, we suppose that the uniform algebra $A$ is self-adjoint, in other words,
that $A$ is a commutative C*-algebra.
The conclusion here is even stronger than the one that we seek.

\begin{theorem}\label{T:vN}
Let $A$ be a commutative C*-algebra,
let $\theta:A\to M_n(\CC)$ be a continuous homomorphism,
and let $\alpha:A\to A$ be a unital antilinear contraction.
If $\theta_\alpha$, defined by \eqref{E:thetaalpha}, satisfies $\|\theta_\alpha\|\le1$, then
$\theta$ is a $*$-homomorphism, and in particular $\|\theta\|=1$.
\end{theorem}

Theorems~\ref{T:CGL} and \ref{T:vN} will be proved and further discussed in \S\ref{S:CGL} and \S\ref{S:vN} respectively.
Finally, in \S\ref{S:CB}, we shall recast the discussion in the setting of completely bounded maps and operator algebras, and highlight the connection with dilation theory.


\section{Proof of Theorem~\ref{T:basic}}\label{S:basic}

The proof that follows is based on an argument from \cite{RS18}.

\begin{proof}[Proof of Theorem~\ref{T:basic}]
For $f\in A$, we have
\begin{align*}
\|\theta(f)\|^4
&=\|\theta(f)\theta(f)^*\theta(f)\theta(f)^*\|\\
&=\|\theta(f)\bigl(\theta(f)+\theta(\alpha(f))^*\bigr)^*\theta(f)\theta(f)^*-\theta(f)\theta(\alpha(f))\theta(f)\theta(f)^*\|\\
&\le\|\theta(f)\|\|2\theta_\alpha(f)^*\|\|\theta(f)\|\|\theta(f)^*\|+\|\theta(f\alpha(f)f)\|\|\theta(f)\|\\
&\le 2\|\theta\|^3\|f\|^4+\|\theta\|^2\|f\alpha(f)f\|\|f\|\\
&\le 2\|\theta\|^3\|f\|^4+\|\theta\|^2\|f\|^4.
\end{align*}
(Here the second inequality used $\|\theta_\alpha\|\le1$, and the last one used the fact that $\alpha$ is a contraction.)
Therefore $\|\theta\|^4\le 2\|\theta\|^3+\|\theta\|^2$, whence $\|\theta\|\le 1+\sqrt2$.

The sharpness of the constant $(1+\sqrt{2})$ follows immediately from the example exhibited in
\cite[\S3]{RS18}.
\end{proof}

\begin{remarks}
(i) The proof of Theorem~\ref{T:basic} makes it clear that the result remains true if $A$ is an arbitrary 
Banach algebra and $M_n(\CC)$ is replaced by an arbitrary C*-algebra. 

(ii) As remarked in the introduction, in the example showing sharpness, we have $\alpha(1)=-1$.
\end{remarks}


\section{Proof of Theorem~\ref{T:link}}\label{S:link}

The proof that follows is based on arguments from \cite{CP17}.

\begin{proof}[Proof of Theorem~\ref{T:link}]
Suppose  that Conjecture~\ref{T:conj} holds when $A=A(\DD)$ and $n=n_0$.
Let $T\in M_{n_0}(\CC)$. We shall prove that \eqref{E:Crouzeix} holds for all polynomials $p$.

Let $\Omega$ be a bounded convex open neighbourhood of $W(T)$ with smooth boundary.
Once again, we denote by $A(\Omega)$ the 
uniform algebra of continuous functions on $\overline{\Omega}$
that are holomorphic on $\Omega$.

As is shown in \cite[\S2]{CP17},
the containment $W(T)\subset\Omega$ 
is reflected in the fact that the operator-valued measure on $\partial\Omega$
used in defining $h(T)$ for $h\in A(\Omega)$, namely
\[
\frac{1}{2\pi i}(\zeta I-T)^{-1}\,d\zeta,
\]
has positive real part.
This quickly  leads to the estimate
\begin{equation}\label{E:f+g*}
\|h(T)+((\cC_\Omega\overline{h})(T))^*\|\le 2\sup_{\partial\Omega}|h|
\quad(h\in A(\Omega)),
\end{equation}
where $\cC_\Omega\overline{h}$ denotes the Cauchy transform of $\overline{h}$ relative to $\Omega$,
defined by
\[
(\cC_\Omega\overline{h})(z):=\frac{1}{2\pi i}\int_{\partial\Omega}\frac{\overline{h(\zeta)}}{\zeta-z}\,d\zeta
\quad(z\in\Omega).
\]
We remark also that, since $\Omega$ is convex,
the map $h\mapsto \cC_\Omega\overline{h}$ is a  contraction of $A(\Omega)$ into itself 
(see \cite[Lemma~2.1]{CP17}).

By the Riemann mapping theorem, there exists a conformal mapping $\phi$ of $\Omega$ onto the unit disk $\DD$,
and by Carath\'eodory's theorem, $\phi$ extends to a homeomorphism of $\overline{\Omega}$ onto $\overline{\DD}$.
For $f\in A(\DD)$, define
\begin{align*}
\theta(f)&:=(f\circ\phi)(T),\\
\alpha(f)&:=\cC_\Omega (\overline{f\circ\phi})\circ\phi^{-1}.
\end{align*}
Then $\theta:A(\DD)\to M_{n_0}(\CC)$ is a continuous homomorphism, 
and $\alpha:A(\DD)\to A(\DD)$ is a unital, antilinear contraction. Further, by \eqref{E:f+g*}, for all $f\in A(\DD)$,
\[
\|\theta(f)+(\theta(\alpha(f))^*\|=\|(f\circ\phi)(T)+\cC_{\Omega}(\overline{f\circ\phi})(T)^*\|\le 2\sup_{\partial\Omega}|f\circ\phi|=2\|f\|_{A(\DD)},
\]
so $\|\theta_\alpha\|\le1$.
Using the supposed truth of Conjecture~\ref{T:conj}, we deduce that $\|\theta\|\le 2$. In particular, if $p$ is a polynomial, then
\[
\|p(T)\|=\|\theta(p\circ\phi^{-1})\|\le 2\|p\circ\phi^{-1}\|_{A(\DD)}=2\sup_{\Omega}|p|.
\] 
Finally, by taking a sequence $\Omega_n$ of smoothly bounded convex neighbourhoods of $W(T)$ shrinking to $W(T)$, 
we deduce that
\[
\|p(T)\|\le 2\sup_{W(T)}|p|,
\] 
as required.
\end{proof}


\section{Proof of Theorem~\ref{T:CGL}}\label{S:CGL}

Theorem~\ref{T:CGL} is a special case of the following, more general result.

\begin{theorem}\label{T:maeva}
Let $A$ be a uniform algebra and let $\theta:A\to M_n(\CC)$ be a continuous homomorphism.
If $\|\theta\|>1$, then, for all $\beta$ in the dual of $A$,
\[
\|\theta+\beta I\|\ge \|\theta\|.
\]
\end{theorem}

The proof of Theorem~\ref{T:maeva} is based on the following lemma,
due to Crouzeix, Gilfeather and Holbrook \cite{CGH14}. 
We write $\sigma(T)$ for the spectrum of $T$.

\begin{lemma}\label{L:xTx=0}
Let $T$ be an $n\times n$ matrix such that $\sigma(T)\subset\overline{\DD}$ and $\|T\|>1$.
Suppose that $\|\phi(T)\|\le\|T\|$ for all automorphisms $\phi$ of $\DD$.
Let $x$ be a unit vector in $\CC^n$ on which $T$ attains its norm.
Then 
\begin{equation}\label{E:claim}
\langle x,Tx\rangle=0.
\end{equation}
\end{lemma}

For completeness, we include a short proof.

\begin{proof}
For $w\in\DD$, let $\phi_w(z):=(z-w)/(1-\overline{w}z)$.
Then
\[
\phi_w(T)=(T-wI)(I-\overline{w}T)^{-1}=T-wI+\overline{w}T^2+O(|w|^2)
\quad(w\to0).
\]
By hypothesis $\|\phi_w(T)\|\le\|T\|$, so it follows that
\begin{align*}
\|T\|^2&\ge 
\langle \phi_w(T)x,\phi_w(T)x\rangle\\
&=\|Tx\|^2-2\Re\bigl(w\langle x,Tx\rangle\bigr)+ 2\Re\bigl(\overline{w}\langle T^2x,Tx\rangle\bigr)+O(|w|^2)\\
&=\|T\|^2-2\Re\bigl(w\langle (I-T^*T)x,Tx\rangle\bigr)+O(|w|^2).
\end{align*}
Letting $w\to0$, and noting that argument of $w$ is arbitrary, it follows that
\[
\langle (I-T^*T)x,Tx\rangle=0.
\]
Since $T$ attains its norm at $x$, we have $T^*Tx=\|T\|^2x$, whence
\[
(1-\|T\|^2)\langle x,Tx\rangle=0.
\]
Finally, since $\|T\|>1$, we conclude that \eqref{E:claim} holds.
\end{proof}

\begin{remark}
The ideas behind this lemma are developed further in \cites{BG20, CGL18, RW20}.
\end{remark}

\begin{proof}[Proof of Theorem~\ref{T:maeva}]
Let $(f_k)$ be a sequence in $A$ such that $\|f_k\|=1$ and $\|\theta(f_k)\|\to\|\theta\|$. 
Replacing $(f_k)$ by a subsequence, we can suppose that $\theta(f_k)\to T$ in $M_n(\CC)$ 
and $\beta(f_k)\to \lambda$.
We then have  
\[
\|T\|=\lim_{k\to\infty}\|\theta(f_k)\|=\|\theta\|, 
\]
and
\[ 
\|T+\lambda I\|=\lim_{k\to\infty}\|\theta(f_k)+\beta(f_k)I\|\le \|\theta+\beta I\|.
\]
Also, since $\theta$ is a homomorphism, we have
$\sigma(\theta(f_k))\subset\sigma(f_k)\subset\overline{\DD}$ for all $k$, 
and hence  $\sigma(T)\subset\overline{\DD}$.

Let $\phi$ be an automorphism of $\DD$.
Then $\phi\circ f_k\in A$ and $\|\phi\circ f_k\|\le1$,
so $\|\theta(\phi\circ f_k)\|\le\|\theta\|$.
On the other hand, as $\theta$ is a homomorphism, 
we have $\theta(\phi\circ f_k)=\phi(\theta(f_k))\to \phi(T)$ as $k\to\infty$.
Therefore $\|\phi(T)\|\le\|\theta\|=\|T\|$. 
Thus Lemma~\ref{L:xTx=0} applies to $T$.

Let $x\in\CC^n$ be a vector such that
$\|x\|=1$ and $\|Tx\|=\|T\|$.  By Lemma~\ref{L:xTx=0},
if $\|T\|>1$, then $\langle x,Tx\rangle=0$, so
\[
\|T\|^2=\langle Tx,Tx\rangle
=\langle (T+\lambda I)x,Tx\rangle 
\le \|T+\lambda I\|\|\|T\|
\le \|\theta+\beta I\|\|T\|.
\]
Hence $\|T\|\le \|\theta+\beta I\|$. 
As $\|\theta\|=\|T\|$, we deduce that $\|\theta\|\le\|\theta+\beta I\|$.
\end{proof}

Let $\rho\ge1$. We say that an operator $T$ on a Hilbert space $\cH$ is in the class $\cC_\rho$
if $T$ has a unitary $\rho$-dilation, i.e., if there exists a Hilbert space $\cK$ containing $\cH$ as a subspace
and a unitary operator $U$ on $\cK$ such that
\[
T^m=\rho P_\cH U^m|_\cH \quad(m\ge1),
\]
where $P_\cH:\cK\to\cH$ denotes the orthogonal projection of $\cK$ onto $\cH$.
It is known that $\cC_1$ is precisely the set of contractions on $\cH$ \cite{Na53},
and that $\cC_2$ is the set of operators whose numerical range is contained in $\overline{\DD}$
\cites{NF66, Be65}.
The following result is due to Okubo and Ando \cite{OA75}.

\begin{theorem}\label{T:OA}
For each $\rho\ge1$, the closed unit disk is a $\rho$-spectral set for all matrices in the class $\cC_\rho$.
\end{theorem}

Theorem~\ref{T:maeva} permits us to give a new and simple proof of this result.

\begin{proof}[Proof of Theorem~\ref{T:OA}]
Fix $\rho\ge1$ and let $T$ be an $n\times n$ matrix in the class $\cC_\rho$.
Then there exists a unitary $\rho$-dilation $U$ of $T$.
Let $p$ be a polynomial. Writing $q:=p-p(0)$,  we then have
\begin{align*}
p(T)&=q(T)+p(0)I\\
&=\rho P_{\CC^n}q(U)|_{\CC^n}+p(0)I\\
&=\rho P_{\CC^n}p(U)|_{\CC^n}+(1-\rho)p(0)I,
\end{align*}
whence
\[
\|p(T)+(\rho-1)p(0)I\|= \rho\|P_{\CC^n}p(U)|_{\CC^n}\|\le\rho \|p(U)\|\le \rho\sup_{\overline{\DD}}|p|.
\]
Defining $\theta(p):=p(T)$ and $\beta(p):=(\rho-1)p(0)$, we see that both $\beta$ and $\theta$ extend by
continuity to the whole disk algebra $A(\DD)$ and that the extended maps satisfy $\|\theta+\beta I\|\le \rho$.
By Theorem~\ref{T:maeva}, it follows that $\|\theta\|\le \rho$, in other words, 
that $\overline{\DD}$ is a $\rho$-spectral set for $T$.
\end{proof}

\begin{remarks}
(i) Taking $\rho=2$ in Theorem~\ref{T:maeva}, we recover the fact that $\overline{\DD}$ is a $2$-spectral set for all matrices $T$ whose numerical range is contained in $\overline{\DD}$, i.e., that Crouzeix's conjecture holds for $\overline{\DD}$.

(ii) Okubo and Ando established their result for general Hilbert space operators. Although Theorem~\ref{T:maeva} was proved for matrices, a standard argument recovers the more general statement of Okubo and Ando; see for instance the proof of Theorem~2 in \cite{Cr07}.
\end{remarks}


\section{Proof of Theorem~\ref{T:vN}}\label{S:vN}

Theorem~\ref{T:vN} is a consequence of another result,
which may be of interest in its own right.

\begin{theorem}\label{T:proj}
Let $A$ be a uniform algebra, let $\theta:A\to M_n(\CC)$ be a continuous homomorphism,
and let $\alpha:A\to A$ be a unital antilinear contraction.
Assume that $\theta_\alpha$, defined by \eqref{E:thetaalpha}, satisfies $\|\theta_\alpha\|\le1$.
If $p$ is a self-adjoint projection in $A$,
then $\theta(p)$ is a self-adjoint projection in $M_n(\CC)$.
\end{theorem}

The proof passes via the following lemma.
Recall that $\sigma(T)$ and $W(T)$ denote the spectrum and numerical range of $T$ respectively.
Also, we write $\conv(S)$ for the convex hull of $S$.

\begin{lemma}\label{L:proj}
Under the hypotheses of Theorem~\ref{T:proj}, 
\[
W(\theta_\alpha(f))\subset\conv(\sigma(f)) \quad(f\in A).
\]
\end{lemma}

\begin{proof}
Let $w\in \CC\setminus\conv(\sigma(f))$. Then there exists a closed disk $\overline{D}(c,r)$
such that $\sigma(f)\subset \overline{D}(c,r)$ and $w\notin \overline{D}(c,r)$. The condition $\sigma(f)\subset \overline{D}(c,r)$
is equivalent to $\|f-c1\|\le r$. The assumptions on $\theta$ and $\alpha$ imply that $\theta_\alpha$ is a unital linear
contraction. Therefore $\|\theta_\alpha(f)-cI\|\le r$. It follows that
\[
|\langle (\theta_\alpha(f)-cI)x,x\rangle| \le r \quad(x\in \CC^n,~\|x\|=1),
\]
whence $W(\theta_\alpha(f))\subset\overline{D}(c,r)$. In particular, $w\notin W(\theta_\alpha(f))$.
\end{proof}

\begin{proof}[Proof of Theorem~\ref{T:proj}]
Since $\theta(A)$ is a commutative subalgebra of $M_n(\CC)$, 
there exists a unitary matrix $U\in M_n(\CC)$ such that $U^*\theta(A)U$
is an algebra of upper triangular matrices (see e.g.\ \cite[Theorem 1]{Ne67}). 
Defining $\tilde{\theta}:A\to M_n(\CC)$ by $\tilde{\theta}(f):=U^*\theta(f)U$
gives a continuous homomorphism that still satisfies
$\|\tilde{\theta}_\alpha\|\le1$,
and if the conclusion of the theorem holds for $\tilde{\theta}$, then it also holds for $\theta$.
Thus, replacing $\theta$ by $\tilde{\theta}$ at the outset,
we may as well assume that $\theta(A)$ consists of upper triangular matrices. 

Let $p$ be a self-adjoint projection in $A$ and set $u:=2p-1$. 
Then $u$ is self-adjoint and $u^2=1$.
Thus $\theta(u)$ is an upper triangular matrix such that $\theta(u)^2=I$,
which forces it to have the form 
\[
\theta(u)=
\begin{pmatrix}
I &X\\  0 &-I
\end{pmatrix},
\]
for some rectangular matrix $X$. Since $u$ is self-adjoint, 
Lemma~\ref{L:proj} implies that $\theta_\alpha(u)$ is self-adjoint,
which entails that
\[
\theta(\alpha(u))=
\begin{pmatrix}
D &X\\ 0 &-E
\end{pmatrix},
\]
for some real diagonal matrices $D$ and $E$.
The fact that $\theta(u)$ and $\theta(\alpha(u))$ commute implies that $DX-X=X-XE$, whence
\begin{equation}\label{E:commute}
2X-DX-XE=0.
\end{equation}
Also, since $\|\theta_\alpha(u)\|\le \|u\|=1$, we must have
\begin{equation}\label{E:norm2}
\left\| 
\begin{pmatrix}
(I+D)/2 &X/2\\  X^*/2 &-(I+E)/2
\end{pmatrix}
\right\|
\le 1.
\end{equation}

Suppose, if possible, that $X$ has a non-zero entry $x_{jk}$.
Condition \eqref{E:commute}  gives $(2-d_j-e_k)x_{jk}=0$,
where $d_j$ and $e_k$ are the $j$-th and $k$-th entries of $D,E$ respectively.
Therefore $d_j+e_k=2$. On the other hand, condition \eqref{E:norm2} implies that $|1+d_j|/2\le1$ and $|1+e_k|/2\le1$.
Therefore $d_j=e_k=1$. But then the matrix in the left-hand-side of \eqref{E:norm2} contains a $2\times 2$ sub-matrix of the form
\[
\begin{pmatrix}
1 &x_{jk}/2\\
\overline{x}_{jk}/2 &-1
\end{pmatrix}
\]
which must still have norm at most $1$. This forces $x_{jk}=0$, a contradiction. We conclude that $X=0$.
Therefore $\theta(u)$ is a diagonal matrix. Therefore so too is $\theta(p)$. Clearly $\theta(p)^2=\theta(p)$,
so $\theta(p)$ is a self-adjoint projection, as claimed.
\end{proof}

\begin{proof}[Proof of Theorem~\ref{T:vN}]
Suppose first  that $A$ is a commutative von Neumann algebra, 
so $A=L^\infty(\Omega,\mu)$ for some measure space $(\Omega,\mu)$.
Let $f\in A$ with $\|f\|=1$. Let $\{D_1,\dots,D_r\}$ be a partition
of $\overline{\DD}$ into Borel sets of diameter at most $\epsilon$.
Pick an element $d_j\in D_j$ for each $j$ and set $g:=\sum_1^r d_jp_j$,
where $p_j:=1_{f^{-1}(D_j)}$.
Then $\|f-g\|_\infty\le\epsilon$. 
Since the $p_j$ are orthogonal self-adjoint projections, so are the $\theta(p_j)$,
by Theorem~\ref{T:proj}.
Therefore
\[
\|\theta(g)\|=\Bigl\|\sum_1^r d_j\theta(p_j)\Bigr\|=\max_j|d_j|\le 1.
\]
It follows that $\|\theta(f)\|\le1+\|\theta\|\epsilon$.
As $\epsilon$ is arbitrary, $\|\theta(f)\|\le1$.
Therefore $\|\theta\|\le1$. Finally, 
as $\theta(1)=I$,
we in fact have $\|\theta\|=1$. This implies in particular that $\theta$ is positive \cite[Proposition 2.11]{paulsen2002}, so that $\theta$ is a $*$-homomorphism.

Now suppose that $A$ is a general commutative C*-algebra.
Then its double dual $A''$ is a commutative von Neumann algebra \cite[Theorem A.5.6]{BLM2004}.
Since $A$ is weak*-dense in $A''$
and the dualized maps $\theta'':A''\to M_n(\CC)$ and $\alpha'':A''\to A''$
are weak*-continuous,
we have  $\|\alpha''\|=\|\alpha\|=1$ and
$\|\theta''_{\alpha''}\|=\|\theta_\alpha\|\le 2$.
Also $\theta''$ is still a homomorphism. 
By the first part of the proof, we infer that $\theta''$ is a $*$-homomorphism, and thus so is $\theta$.
\end{proof}

\section{A completely bounded version of the conjecture}\label{S:CB}

In this section, we discuss a completely bounded version of the preceding ideas, thus framing them within the context of operator algebras. As we will see below, this allows dilation techniques to enter the picture and to shed new light on the problem. 

Throughout this section, $B(\cH)$ denotes the $C^*$-algebra of bounded linear operators on some Hilbert space $\cH$. An operator algebra is simply a \emph{unital} subalgebra $A\subset B(\cH)$. While it is possible to define these objects without referring to a particular choice of representation on Hilbert space \cite[Corollary 16.7]{paulsen2002}, for our purposes the previous concrete description suffices. We remark that subalgebras of $C^*$-algebras are operator algebras, and in particular so are uniform algebras.

Given a natural number $n\geq 1$, we denote by $M_n(A)$ the algebra of $n\times n$ matrices with entries from $A$, which we view as a subalgebra of bounded linear operators acting on $\cH^{(n)}=\cH\oplus \cH\oplus \ldots \oplus \cH$. In particular, $M_n(A)$ is endowed with a norm under this identification. 
Given a map $\phi:A\to B(\cK)$, for each natural number $n\geq 1$, we may define a map $\phi^{(n)}:M_n(A)\to B(\cK^{(n)})$ as
\[
\phi^{(n)}([a_{ij}])=[\phi(a_{ij})] \quad \bigl([a_{ij}]\in M_n(A)\bigr).
\]
It is readily verified that $\phi^{(n)}$ is linear (respectively, a homomorphism) whenever $\phi$ is. 

If $\phi$ is linear (or anti-linear), we say that $\phi$ is \emph{completely bounded} if the quantity
\[
\|\phi\|_{cb}=\sup_n \|\phi^{(n)}\|
\]
is finite. Furthermore, we say that $\phi$ is \emph{completely contractive} if $\|\phi\|_{cb}\leq 1$.

We can now state the refined version of the main conjecture, which is the central topic of this section.

\begin{conjecture}\label{C:conjCB}
Let $A$ be an operator algebra, let $\theta:A\to M_n(\CC)$ be a unital completely bounded homomorphism,
and let $\alpha:A\to A$ be a unital antilinear complete contraction. 
Define a linear map $\theta_\alpha:A\to M_n(\CC)$ by \eqref{E:thetaalpha}.
If $\|\theta_\alpha\|_{cb}\leq 1$, then $\|\theta\|_{cb}\leq 2$.
\end{conjecture}

We start by showing that the estimate  $\|\theta\|_{cb}\leq 1+\sqrt{2}$ always holds, mirroring the previous setting.

\begin{proposition}\label{P:clever}
Let $A$ be an operator algebra. Let $\theta,\alpha$ and $\theta_\alpha$ be as in Conjecture \ref{C:conjCB}. Then $\|\theta\|_{cb}\leq 1+\sqrt{2}$.
\end{proposition}
\begin{proof} Fix a positive integer $n$.
Remark~(i) following the proof of Theorem~\ref{T:basic} shows that the argument therein applies to the maps $\theta^{(n)}$ and $\alpha^{(n)}$. Hence $\|\theta^{(n)}\|\leq 2$ for every $n\geq 1$, which is equivalent to $\|\theta\|_{cb}\leq 2$.
\end{proof}

We note that $*$-homomorphisms on $C^*$-algebras are always completely contractive, so that Theorem \ref{T:vN} shows that Conjecture \ref{C:conjCB} is verified for commutative $C^*$-algebras. In contrast, we do not know if the proof of Theorem \ref{T:CGL} can be adapted to show that Conjecture \ref{C:conjCB} holds in the case of the disk algebra when $\alpha$ takes values in the scalar multiples of the identity.

In trying to verify Conjecture \ref{C:conjCB} for the disc algebra, the following well-known basic principle is useful.

\begin{lemma}\label{L:simz}
Let $\theta:A(\DD)\to B(\cH)$ be a unital completely bounded homomorphism and let $\psi$ be an automorphism of $\DD$. Then, $\|\theta\|_{cb}\leq 2$ if and only if there is an invertible operator $S\in B(\cH)$ with the property that $S\theta(\psi)S^{-1}$ is a contraction and that
$
\|S\| \|S^{-1}\|\leq 2.
$
\end{lemma}
\begin{proof}
Assuming that $\|\theta\|_{cb}\leq 2$, the existence of $S$ follows at once from Paulsen's similarity theorem \cite{paulsen1984PAMS}. Conversely, given an operator $S$ as in the statement, define $\theta_S:A(\DD)\to B(\cH)$ as
\[
\theta_S(f)=S\theta(f)S^{-1} \quad (f\in A(\DD)).
\]
 For a matrix of polynomials $[p_{ij}]\in M_n(A(\DD))$, we see that
\begin{align*}
\theta^{(n)}_S([p_{ij}\circ \psi])&=[S\theta(p_{ij}\circ\psi)S^{-1}]\\
&=[p_{ij}(S\theta(\psi)S^{-1})]
\end{align*}
so von Neumann's inequality implies that
\[
\|\theta^{(n)}_S([p_{ij}\circ \psi])\|\leq \|[p_{ij}]\|_{M_n(A(\DD))}= \|[p_{ij}\circ \psi]\|_{M_n(A(\DD))}.
\]
We conclude that $\theta_S$ is completely contractive on the dense subalgebra
\[
\{ p\circ \psi:p \text{ polynomial}\},
\]
and thus on the entire algebra $A(\DD)$. Consequently,
\[
\|\theta\|_{cb}\leq \|S\| \|S^{-1}\| \|\theta_S\|_{cb}\leq 2.\qedhere
\]
\end{proof}

We do not know of a good description of the class of operators $T\in B(\cH)$ for which  there is an invertible operator $S\in B(\cH)$ with the property that $STS^{-1}$ is a contraction and that
$
\|S\| \|S^{-1}\|\leq 2.
$
A sufficient condition for this property to hold was discovered in \cite{OA75}. 

We now give an instance where the previous lemma can be exploited.

\begin{proposition}\label{P:singleton}
Let $\theta:A(\DD)\to M_2(\CC)$ be a unital bounded homomorphism and let $\alpha:A(\DD)\to A(\DD)$ be a unital antilinear map. Define $\theta_\alpha:A(\DD)\to M_2(\CC)$ by \eqref{E:thetaalpha}.
Assume that $\|\theta_\alpha\|\le1$ and that the spectrum of $\theta(z)$ consists of a single point in $\DD$. Then $\|\theta\|_{cb}\leq 2$.
\end{proposition}
\begin{proof}
By assumption, there is a $\lambda\in \DD$ such that $\sigma(\theta(z))=\{\lambda\}$. Let $\psi\in A(\DD)$ be defined by
\[
\psi(z):=\frac{z-\lambda}{1-\ol{\lambda}z} \quad (z\in \DD).
\]
Then the spectrum of $\theta(\psi)$ is $\{0\}$. Since $\theta(\psi)$ and $\theta(\alpha(\psi))$ are commuting matrices, we can assume that they are both upper triangular, and that the diagonal of $\theta(\psi)$ consists entirely of zeros. 
Hence there is a $\delta\in \CC$ along with $a,b,c\in \CC$ such that
\[
\theta(\psi)=\begin{pmatrix} 0 & \delta\\ 0 & 0 \end{pmatrix} \quad \text{and} \quad  \theta(\alpha(\psi))=\begin{pmatrix} a &b\\ 0 & c \end{pmatrix}.
\]
Then we have
\[
2\theta_\alpha(\psi)= \begin{pmatrix} \ol{a} & \delta\\ \ol{b} & \ol{c}\end{pmatrix}.
\]
Since $\|\theta_\alpha(\psi)\|\le \|\theta_\alpha\|\le1$, this last matrix has norm at most~$2$,
and in particular  $|\delta|\leq 2$. Let 
\[
S:=\begin{pmatrix} 1/2 & 0\\ 0 & 1\end{pmatrix}.
\] 
Then 
\[
S\theta(\psi)S^{-1}=\begin{pmatrix} 0 & \delta/2\\ 0 & 0 \end{pmatrix},
\]
so $S\theta(\psi)S^{-1}$ is a contraction and $\|S\| \|S^{-1}\|\leq 2$. Lemma \ref{L:simz} finally implies that $\|\theta\|_{cb}\leq 2$.
\end{proof}

We do not know how to handle the case where $\theta(z)$ in the previous theorem is allowed to have arbitrary spectrum. Even in the concrete case of the usual Crouzeix conjecture, the situation gets much more complicated; see \cite[Theorem 5.2]{BCD06}.

As mentioned earlier, the setting of Conjecture \ref{C:conjCB} allows for the use of dilation techniques; this is explored in the next result. 


\begin{theorem}\label{T:equiv}
Let $\theta:A(\DD)\to B(\cH)$ be a unital bounded homomorphism and let $\alpha:A(\DD)\to A(\DD)$ be a unital antilinear map. Define $\theta_\alpha:A(\DD)\to B(\cH)$ by
\[
\theta_\alpha(f):=\frac{1}{2}(\theta(f)+\theta(\alpha(f))^*) \quad (f\in A(\DD)).
\]
Consider the following statements:
\begin{enumerate}[\normalfont(i)]
\item The map $\theta_\alpha$ is completely contractive.
\item There is a unitary operator $U$ on a Hilbert space $\cK$ containing $\cH$ such that 
\[
\theta(f)+\theta(\alpha(f))^*=2P_{\cH}f(U)|_{\cH} \quad (f\in A(\DD)).
\]
\item The operator $\theta(z)$ admits a unitary $2$-dilation.
\item The numerical range of $\theta(z)$ is contained in the closed unit disk.
\item The estimate  $\|\theta\|_{cb}\leq 2$ holds.
\end{enumerate}
Then we have 
\[
(i) \Rightarrow (ii) \quad\text{and}\quad (iii) \Leftrightarrow (iv)\Rightarrow (v).
\]
If we assume in addition that $z^m\in \ker (\theta\circ \alpha)$ for every positive integer $m$, then we have 
\[
(i) \Rightarrow (ii)\Rightarrow (iii) \Leftrightarrow (iv)\Rightarrow (v).
\]
\end{theorem}

\begin{proof}
Assume that (i) holds. Recall that $A(\DD)\subset C(\TT)$.
By combining Arveson's extension theorem \cite[Corollary 7.6]{paulsen2002} with Stinespring's dilation theorem \cite[Theorem 4.1]{paulsen2002}, we obtain a Hilbert space $\cK$ containing $\cH$ and unital $*$-homomorphism $\pi:C(\TT)\to B(\cK)$ such that
\[
\frac{1}{2}\bigl(\theta(f)+\theta(\alpha(f))^*\bigr)=\theta_\alpha(f)=P_{\cH}\pi(f)|_{\cH} \quad (f\in A(\DD)).
\]
Choosing $U:=\pi(z)$, which is unitary, we see that (ii) holds. Furthermore, we see that (ii) implies (iii) provided that $\theta(\alpha(z^m))=0$ for every $m\geq 1$.

The equivalence of (iii) and (iv) is found in  \cite[Theorem I.11.2]{nagy2010}. 

The fact that (iii) implies (v) follows from \cite{OA75} along with Lemma \ref{L:simz}.
\end{proof}

We conclude by remarking  that, when $\alpha(f)=\ol{f(0)}I$, then (i), (ii), (iii) and (iv) above are equivalent; see \cite[Theorem 3.15]{paulsen2002}.


\bibliographystyle{plain}
\bibliography{biblist}

\end{document}